\documentclass [12pt,a4paper]{amsart} 

\usepackage[T1]{fontenc}

\usepackage{graphicx}
\usepackage[utf8]{inputenc} %utilizar acentos normalmente
\usepackage{amssymb,amsmath,amsfonts,amsthm,mathtools,nicefrac}
\usepackage{enumitem}
\renewcommand{\raggedright}{\justifying}							
\usepackage{xcolor}
\usepackage{helvet}							% fonte helvetica
\usepackage{comment}

\usepackage{pgfplots}
\pgfplotsset{compat=1.15}
\usepackage{mathrsfs}
\usetikzlibrary{arrows}
\usepackage{relsize}
 
\usepackage{extarrows}
\usepackage[shortcuts]{extdash}

\newtheorem{theorem}{Theorem}[section]
\newtheorem{lemma}[theorem]{Lemma}
\newtheorem{proposition}[theorem]{Proposition}
\newtheorem{corollary}[theorem]{Corollary}

\theoremstyle{definition}
\newtheorem{definition}[theorem]{Definition}

\theoremstyle{remark}

\newtheorem*{notation*}{Notation}
				
\newcommand{\Hyp}{\mathbb{H}}
\newcommand{\bH}{\mathbb{H}}
\newcommand{\cH}{\mathcal{H}}
\newcommand{\cL}{\mathcal{L}}
\newcommand{\rM}{\mathrm{M}}

\newcommand{\cO}{\mathcal{O}}
\newcommand{\rP}{\mathrm{P}}
\newcommand{\Q}{\mathbb{Q}}
\newcommand{\R}{\mathbb{R}}
\newcommand{\id}{\mathrm{id}}
\newcommand{\SL}{\mathrm{SL}}
\newcommand{\PSL}{\mathrm{PSL}}

\newcommand{\tr}{\mathrm{tr}}

\DeclareMathOperator{\arcosh}{arcosh}

\newcommand{\Aff}{\mathrm{Aff}}%
\newcommand{\C}{\mathbb{C}}
\newcommand{\cT}{\mathcal{T}}%
\newcommand{\rd}{\mathrm{d}}%
\newcommand{\cM}{\mathcal{M}}%
\newcommand{\Mod}{\mathrm{Mod}}%
\newcommand{\SO}{\mathrm{SO}}%
\newcommand{\Stab}{\mathrm{Stab}}%
\newcommand\norm[1]{\lVert#1\rVert}%

\def\tareesidedbox#1{\setbox0=\hbox{$#1$}\dimen0=\wd0 \advance\dimen0 by3pt\rlap{\hbox{\vrule height9pt width.4pt depth2pt \kern-.4pt\vrule height9.4pt width\dimen0 depth-9pt\kern-.4pt \vrule height9pt width.4pt depth2pt}} \relax \hbox to\dimen0{\hss$#1$\hss}}

\usepackage{hyperref}

\usepackage{enumitem}
 \title[Multiplicities in spectra of semi-arithmetic surfaces]{On multiplicities in length spectra of semi-arithmetic hyperbolic surfaces}

\author[M. Belolipetsky]{Mikhail Belolipetsky}
\address{IMPA, Estrada Dona Castorina 110, 22460-320 Rio de Janeiro, Brazil}
\email[]{mbel@impa.br}

\author[G. Cosac]{Gregory Cosac} 
\address{Departamento de Matem\'atica Aplicada, IME - Universidade de S\~ao Paulo, Rua do Mat\~ao 1010, 05508-090 – São Paulo, SP, Brazil}
\email{cosac@ime.usp.br}

\author[C. D\'oria]{Cayo D\'oria}
\address{UFS, Departamento de Matem\'atica - Av. Marcelo D\'eda Chagas s/n, 49100-000. S\~ao Crist\'ov\~ao, Brazil}
\email{cayo@mat.ufs.br}

\author[G. Teixeira Paula]{Gisele Teixeira Paula}
\address{UFPR, Centro Politécnico - Av. Cel. Francisco H. dos Santos 100, 81530-000. Curitiba, Brazil}
\email[]{giseleteixeira@ufpr.br}

\begin{document}
 
\begin{abstract}
We show that semi-arithmetic surfaces of arithmetic dimension two which admit a modular embedding have exponential growth of mean multiplicities in their length spectrum. Prior to this work large mean multiplicities were rigorously confirmed only for the length spectra of arithmetic surfaces.
We also discuss the relation of the degeneracies in the length spectrum and quantization of the Hamiltonian mechanical system on the surface.
\end{abstract}

\maketitle

\section{Introduction}
Let $S = \Gamma\backslash\Hyp$ be a hyperbolic surface, possibly with elliptic singularities and with cusps. The prime geodesic theorem, which goes back to the classical work of Huber and Selberg, asserts that the number of closed geodesics on $S$ of length at most $\ell$ grows exponentially: 
$$\mathcal{N}(\ell) \sim \dfrac{e^\ell}{\ell} \text{ as } \ell\to\infty.$$ 

% H.Huber, Math. Annalen 138,1(1959)

Distribution and multiplicity of the lengths of closed geodesics (periodic orbits) on surfaces have been of interest for theoretical physics in connection with the phenomenon of quantum chaos. In 1980s, Aurich and Steiner experimentally observed that some surfaces have extremely large multiplicities in their geodesic length spectrum. This result is essentially the starting point for our present work so let us give some definitions and state it precisely. 

\iffalse
We denote by $\mathcal{N}'(\ell)$ the number of distinct lengths of geodesics in the interval $(0,\ell]$.  The \emph{mean multiplicity} $\langle g(\ell) \rangle$ is defined as a continuous function of $\ell$ that is an average order of the multiplicity of lengths of geodesics on $S$:
$$\mathcal{N}(\ell) \sim \int_0^\ell \mathcal{N'}(x)\langle g(x) \rangle dx.$$
\fi

Let $\{0 < \ell_1 < \ell_2 < \cdots \}$ be the geodesic length set (that is, without multiplicities) of $S$ with counting function
\begin{align*}
    \mathcal{N}'(\ell) = \#\{n \, ; \,\ell_n \leq \ell\}.
\end{align*}
Denote by $g(\ell_n)$ the multiplicity of $\ell_n$ in the geodesic length spectrum of $S$. The \emph{mean multiplicity} $\langle g(\ell) \rangle$ is defined to be a continuous function of $\ell$ that is an \emph{average order} of $g$, meaning:

\begin{align*}
    \mathcal{N}(\ell) = \sum_{\ell_n \leq \ell} g(\ell_n) \sim \sum_{\ell_n \leq \ell} \langle g(\ell_n)\rangle = \int_0^\ell \langle g(x)\rangle \, d\mathcal{N'}(x),
\end{align*}
as $\ell \to \infty$. The right-hand side is the Riemann--Stieltjes integral of the mean multiplicity with respect to $\mathcal{N'}$.

We refer to \cite[Section~3.2]{Bol93_long} for more details regarding this notion. We will say that $S$ has \emph{exponential growth of the mean multiplicities} (EGMM) if for sufficiently large $\ell$ its mean multiplicity satisfies 
$$\langle g(\ell) \rangle \geq e^{C\ell}\text{ for some constant } C = C(S) > 0.$$

 In their work, Aurich and Steiner detected a strong exponential growth of the mean multiplicities in the spectrum of some particular surfaces \cite{AuSt88}, by showing that 
$\langle g(\ell) \rangle \sim c \dfrac{e^{\ell/2}}{\ell}\text{ as } \ell\to\infty.$
It was later noticed in \cite{Bog92} that EGMM tends to occur in the spectrum of \emph{arithmetic surfaces}, i.e. when $\Gamma$ is an arithmetic Fuchsian group. A short proof of this fact follows from the work of Luo and Sarnak \cite{LuoSarnak94} (see also \cite{Bol93} for a detailed derivation which, in addition, yields the value for the multiplicative constant in the asymptotic growth formula).

The proofs in \cite{LuoSarnak94} and \cite{Bol93} use the bounded clustering property of the length spectrum of arithmetic surfaces. More precisely, we say that a group $\Gamma < \SL(2, \R)$ satisfies the \emph{bounded clustering} property (B-C) if there exists a constant $c = c(\Gamma) < \infty$ such that 
$$\#(\mathcal{L}(\Gamma) \cap [N-1, N]) \leq c\text{ for all } N\in \mathbb{N}, $$
where $\mathcal{L}(\Gamma) = \{ |\tr(\gamma)| \mid \gamma\in\Gamma \}$. The absolute values of traces of hyperbolic isometries determine the lengths of corresponding geodesics on the quotient surface (see Section~\ref{sec:prelim}), hence we defined the bounded clustering in the length spectrum of $S$.  
Sarnak conjectured that the \mbox{B-C} property is characteristic for arithmetic Fuchsian groups, which was later proved in the non-cocompact case by Geninska and Leuzinger \cite{GenLeu08}. Around the time of Sarnak's conjecture it was assumed and stated in some papers that the EGMM property should also imply arithmeticity (see e.g. \cite{Bol93_long}). Later on, however, physicists found experimental evidence and heuristic arguments showing that some non-arithmetic hyperbolic surfaces also tend to have EGMM, although with a smaller exponent \cite{Bog97, Bog04}. In this paper, we prove this fact. To the best of our knowledge, this is the first time when exponential growth of the mean multiplicities in the length spectrum is rigorously confirmed in the non-arithmetic case.

Our main result is the following
\begin{theorem}\label{thm}
 All semi-arithmetic Fuchsian groups which admit a modular embedding and have arithmetic dimension at most \(2\) have exponential growth of the mean multiplicity in the length spectrum.   
\end{theorem}
% Our main results are given in Theorems \ref{thm1} and \ref{thm2} below that state that all cocompact (resp., non-cocompact) semi-arithmetic Fuchsian groups which admit a modular embedding and have arithmetic dimension at most \(2\) have exponential growth of the mean multiplicity in the length spectrum. 
We will give definitions of (semi-)arithmeticity, modular embedding and arithmetic dimension in Section~\ref{sec:prelim}. The proof is different for cocompact and non-cocompact Fuchsian groups, which are treated in Theorem~\ref{thm1} and Theorem~\ref{thm2}, respectively. The proof of the first theorem uses the Schwarz--Pick lemma and some arithmetic properties of semi-arithmetic groups. The latter imply a weaker form of the B-C property which is sufficient for our purpose
(see Lemma~\ref{lem B-C}). This argument does not generalize directly to the non-compact surfaces. In order to prove Theorem~\ref{thm2} we first apply the generalized Kirszbraun theorem of \cite{LangSchr97}, which allows us to work with the compact core of the surface $S$. With this we are able to replicate the proof of Theorem~\ref{thm1}. The details are given in Section~\ref{sec:proofs}.

For arithmetic dimension $1$ our theorem refers to the well known case of the arithmetic Fuchsian groups, the new result is about groups whose arithmetic dimension is equal to $2$. In Section \ref{290625.1} we list examples of such groups. By  Cohen and Wolfart \cite{Cohen90}, all Fuchsian triangle groups admit modular embeddings. Nugent and Voight showed in \cite{Nugent17} that there exist only finitely many triangle groups of a given arithmetic dimension (see also \cite{BCDT25} for a different proof of a more general result). The method of \cite{Nugent17} is effective and allows us to produce a list of such groups, for small dimensions the results are presented in Section~6 of their paper. The list of triangle groups of arithmetic dimension $2$ consists of $148$ cocompact and $16$ non-cocompact groups. The cocompact part of the list includes four groups from the Hilbert series (cf. \cite{McMullen23}). The non-cocompact part includes a remarkable non-arithmetic Hecke group $\Delta(2,5,\infty)$ which was closely studied in a recent paper by McMullen \cite{McMullen22}. The previous work of McMullen also provides an infinite set of examples of non-cocompact semi-arithmetic Fuchsian groups of arithmetic dimension $2$ which admit a modular embedding and are not commensurable with the triangle groups \cite{McMullen03}. 

We see that our results provide many interesting examples of non-arithmetic Fuchsian groups with EGMM, but there is a natural question that remains open: What can we say about semi-arithmetic groups of arithmetic dimension bigger than $2$? Our method breaks here because the estimate we get for the number $\mathcal{N}'(\ell)$ of geodesics of different lengths has the same growth rate as the total number of geodesics $\mathcal{N}(\ell)$. From the first view, it appears as a shortfall of the method, however, a similar problem was previously encountered in computer experiments of Bogomolny and Schmit \cite{Bog04}. They observed that for these groups the increase of multiplicities numerically is too small to be detected from direct calculations of periodic orbits that they performed. It leaves open the problem about growth of mean multiplicities for general semi-arithmetic Fuchsian groups. Tackling this problem might require new ideas, so far we do not have even a working conjecture that would apply to this case. It is worth noting that generic Fuchsian groups are expected to have bounded mean multiplicity (cf. \cite[p.~247]{Bog97}). 

\medskip

\noindent
\textbf{Acknowledgments.} We thank Jens Marklof and Curtis McMullen for their interest in this work. We also thank the anonymous referee for careful proofreading and helpful comments. 
The work of M.B. is partially supported by the FAPERJ grant E-26/204.250/2024,  the Institut Henri Poincar\'e (UAR 839 CNRS-Sorbonne Universit\'e), and LabEx CARMIN (ANR-10-LABX-59-01).
G.C. is grateful for the grant 2024/01650-3, S\~ao Paulo Research Foundation (FAPESP). The authors C.D. and G.T.P. are grateful for the support of CNPq Grant 408834/2023-4. 

\section{Preliminaries}\label{sec:prelim}
\subsection{}
Let $\bH$ denote the hyperbolic plane, whose full group of orientation-preserving isometries is well known to be the group $\PSL(2,\R)$. Elements of $\PSL(2,\R)$ are classified in terms of their fixed points as elliptic, parabolic or hyperbolic. The last of them are those elements for which the absolute values of the traces of corresponding matrices in $\SL(2,\R)$ are bigger than 2.
The trace of a hyperbolic element  $\gamma \in \PSL(2,\R)$ is defined up to a sign and determines the translation length $\ell(\gamma)$ of the isometry, which is the length of the corresponding closed geodesic on the quotient surface. This relation is given by the formula 
\begin{equation}\label{eq:trace}
\ell(\gamma) = 2 \arcosh \left( \frac{|\tr(\gamma)|}{2} \right).
\end{equation}

Let $\Gamma < \PSL(2,\R)$ be a finitely generated Fuchsian group. The \emph{invariant trace field} of $\Gamma$ is the field generated by the traces of all elements of the group $\Gamma^{(2)} = \langle \gamma^2 \mid \gamma \in \Gamma \rangle$ over the rational numbers. This field will be denoted by \(K\)  and, as its name suggests, is an invariant of the commensurability class of $\Gamma$ (see \cite[Chapter~3]{maclachlan2003arithmetic}).

\begin{definition}
    Let $\Gamma$ be a Fuchsian group of finite covolume. We say that $\Gamma$ is \emph{semi-arithmetic} if $K$ is a totally real number field and the traces of elements of $\Gamma$ are algebraic integers.
\end{definition}

This is a natural generalization of \emph{arithmetic Fuchsian groups}. Indeed, by using a characterization given by Takeuchi, arithmetic Fuchsian groups must also satisfy the extra condition that every non-trivial  embedding $K \hookrightarrow \R$ maps the set of traces of elements of $\Gamma^{(2)}$ to a bounded subset $[-2,2]\subset\R$ (cf. \cite{Takeuchi75}). In the semi-arithmetic case we allow $r \ge 1$ embeddings of $K$ for which the images of the set of traces are unbounded. The number $r$ is called the \emph{arithmetic dimension} of $\Gamma$ (cf. \cite{Nugent17}). Notice that arithmetic Fuchsian groups have arithmetic dimension $1$, and in general we have $1\leq r \leq d$, with $d=[K:\Q]$ denoting the degree of the invariant trace field.

We can also characterize these groups in terms of quaternion algebras as it is described, for instance, in \cite{BCDT25}. Let us briefly recall the construction, later on we will use it in the definition of a modular embedding.

Let $\sigma_1 = \mathrm{id}, \sigma_2, \ldots, \sigma_d$ be the embeddings of $K$ into $\R$. We  assume that the associated quaternion algebra $A/K$ splits over $\sigma_1 = \mathrm{id}$, \ldots, $\sigma_r$. Thus we have an isomorphism
\begin{align}\label{isomosphism_algebra}
    A\otimes_\Q \R \cong  \rM(2,\R)^{r} \times \cH^{d-r},
\end{align}
which maps $x\otimes y$ to $(y\rho_1(x),\dots,y\rho_d(x))$, where $\cH$ is Hamilton's quaternion algebra. Notice that we have
\begin{align*}
    \tr\,\rho_i(x) = \sigma_i(\tr(x)) \quad \text{ and } \quad \mathrm{det}\,\rho_i(x) = \sigma_i(\mathrm{norm}(x)),
\end{align*}
for all $x\in A$ and all $i = 1, \ldots, d$.

Let $A^1$ be the subgroup of elements of $A$ of norm $1$. Then $\rho = (\rho_1,\dots,\rho_d)$ restricts to an embedding
\begin{align*}
    \rho: A^1 \hookrightarrow \SL(2,\R)^r
\end{align*}
and for each $x\in A^1$, $\rho(x) = (\rho_1(x),\ldots,\rho_r(x))$ acts on $\bH^r$ componentwise via M\"obius transformations.

If $\cO$ is an order in $A$, then by the Borel and Harish-Chandra theorem $\rho(\cO^1)$ is a \emph{lattice} in $\SL(2,\R)^r$, i.e. it is a discrete subgroup of finite covolume.

\begin{definition}[cf. {\cite{SW00}}]\label{arithactingonH2}
    A subgroup $\Delta$ of $\PSL(2,\R)^r$ will be called an \emph{arithmetic group acting on} $\bH^r$ if it is commensurable to some $\rP\rho_1(\cO^1)$ as above. If $\Delta$ is a finite index subgroup of $\rP\rho_1(\cO^1)$, we say it is \emph{derived from a quaternion algebra}.
\end{definition}

Any group derived from a quaternion algebra as above acts on $\bH^r$ in a natural way, since it can be embedded in $\SL(2,\R)^r$ by the map $\rho \circ \rho_1^{-1}$. 

The following proposition provides a general characterization of semi-arithmetic Fuchsian groups in terms of the preceding discussion:

\begin{proposition}[{\cite[Proposition 1]{SW00}}]
    A Fuchsian group $\Gamma$ of finite covolume is semi-arithmetic of arithmetic dimension $r$ if and only if it is commensurable to a subgroup of an arithmetic group $\Delta$ acting on $\bH^r$.
\end{proposition}

We observe in \cite{BCDT25} that, if $\Gamma$ is a semi-arithmetic group of arithmetic dimension $r$, then \(\Gamma^{(2)}\) is always derived from a quaternion algebra. By this we mean that $\Gamma^{(2)} < \Delta$, for some group $\Delta$ acting on $\bH^r$ and derived from a quaternion algebra. We call this group the \emph{ambient arithmetic group} of \(\Gamma\).

\subsection{}

From now on we restrict to the case $r = 2$, which is our primary interest in this paper. 
Let $\Gamma$ be a semi-arithmetic Fuchsian group, with arithmetic dimension 2. 

Notice that the group $\Gamma^{(2)}$ is derived from a quaternion algebra and has a family of embeddings $f: \Gamma^{(2)} \to \SL(2,\R)^2$ given by restriction of  $\rho\circ\rho_1^{-1}$  for each choice of $\rho = (\rho_1, \rho_2)$ as above.

\begin{definition}
    We say that $\Gamma$ \emph{admits a modular embedding} 
    if $f$ can be chosen so that there exists an $f$-equivariant holomorphic (or anti-holomorphic) function $\tilde{F}: \bH \to \pm\bH\times \pm\bH$, that is, a function $\tilde{F} = (\tilde{F}_1, \tilde{F}_2)$ satisfying
    \begin{align*}
        \tilde{F}(\gamma \cdot z) = f(\gamma)\cdot \tilde{F}(z)
        = (\gamma\cdot\tilde{F}_1(z) \, , \, \rho_2\circ\rho_1^{-1}(\gamma)\cdot\tilde{F}_2(z)),
    \end{align*}
    for all $z\in \bH$ and all $\gamma \in \Gamma^{(2)}$, with each $\tilde{F}_i$ holomorphic or anti-holomorphic.
\end{definition}

We conclude this section with some geometric properties of semi-arithmetic groups. Let $\Delta < \SL(2,\R)^2$ be the ambient arithmetic group of $\Gamma$ and 
\begin{align*}
    S = \Gamma^{(2)}\backslash\Hyp \hookrightarrow X = \Delta\backslash(\pm\Hyp\times\pm\Hyp)
\end{align*}
the associated embedding of quotient spaces. We have:
\begin{itemize}
    \item[(i)] If $\Gamma$ admits a modular embedding, then $S$ is totally geodesic  in $X$ with the Kobayashi metric (see \cite[Section~6]{McMullen23}).
    \item[(ii)] The group $\Gamma^{(2)}$ is Zariski dense in $\SL(2,\R)^2$ (see \cite[Section~5]{McMullen23}).
    \item[(iii)] If we consider the natural Riemannian metric on \(-\Hyp\) which makes complex conjugation an isometry, the surface $S$ is not totally geodesic in $X$ with the locally symmetric metric induced by the projection 
    % $\tilde{X} \subset (\pm\Hyp\times\pm\Hyp) \to X$
    (otherwise, the subgroup $\Gamma$ would be arithmetic by Bergeron--Clozel \cite[Proposition~15.2.2]{BergeroClozel05} and the remark immediately after the proof of the proposition addressing the non-cocompact case).  In particular, this shows that the map $F$ is not an automorphism.    
\end{itemize}

\section{Results}\label{sec:proofs}

Let \(\Gamma\) be a cocompact semi-arithmetic Fuchsian group of arithmetic dimension \(2\) which admits a modular embedding. In what follows, we denote by $K$ the invariant trace field of $\Gamma$ and by $\cO$ the ring of integers of $K$. The field \(K\) is a  totally real number field of degree \(d \ge 2\) whose embeddings into $\R$ will be denoted by $\sigma_1 = \id, \sigma_2 = \sigma, \sigma_3, \ldots,\sigma_d$. 

Let \(F:\Hyp \to \pm \Hyp\) be a holomorphic (or anti-holomorphic) map provided by the modular embedding of $\Gamma$. It satisfies 
\begin{equation}\label{equivariance}
  F(\gamma \cdot z) = \gamma^\sigma \cdot F(z)  
\end{equation}
for all \(\gamma \in \Gamma^{(2)}\), where \(\gamma^\sigma\) denotes $\rho_2\circ\rho_1^{-1}(\gamma)$. The notation is motivated by the fact that the trace and norm of \(\gamma^\sigma\) are equal to the trace and norm of the $\sigma$-conjugate of $\gamma$. In particular, \(\gamma^\sigma\) is parabolic if and only if \(\gamma\) is parabolic.

Notice that if the group $\Gamma$ is derived from a quaternion algebra, then \eqref{equivariance} applies to all elements $\gamma\in\Gamma.$

\begin{lemma}\label{norm bound}
Let \(\sigma: K \to \R\) be the unique non-trivial embedding with \(\sigma(\tr(\Gamma^{(2)})) \not\subset [-2,2] \). Then there exists \(0 < \delta=\delta(\Gamma) < 1\)  such that
\begin{equation}\label{strict bound}
|\sigma(\tr(\gamma))| < 2|\tr(\gamma)|^{1-\delta}
\end{equation}
  for  all  hyperbolic elements \(\gamma \in \Gamma^{(2)}\) (or all hyperbolic $\gamma \in \Gamma$ in case the group $\Gamma$ is derived from a quaternion algebra). In particular,
\begin{equation}\label{eq norm bound} 
|\mathrm{N}_{K \mid \Q}(\tr(\gamma))| < 2^{d-1}|\tr(\gamma)|^{2-\delta}
\end{equation}
for all  such elements \(\gamma\).
\end{lemma}
\begin{proof} We can assume that $\Gamma$ is derived from a quaternion algebra. 
    Let \(\Omega\) be a fundamental domain  for \(\Gamma\) contained in a compact set. The map \(F\) defined above is holomorphic (or anti-holomorphic) and not an automorphism (see property (iii) from the previous section or \cite[Theorem~3 and Corollary~5]{SW00}).  By the Schwarz--Pick lemma, there exists \(0<\delta<1\) such that \[\sup_{z \in \overline{\Omega}} |DF_z|=1-\delta.\]
    Hence, by \eqref{equivariance} we have \(\sup_{z \in \Hyp} |DF_z|=1-\delta\). Notice that cocompactness of $\Gamma$ is essential for this argument because in the presence of cusps we may have $\sup_{z \in \overline{\Omega}} |DF_z|=1$.  
    
    Let \(\gamma \in \Gamma\) be a hyperbolic element with displacement 
    \[\ell=d(z,\gamma z)=2\arcosh\left(\frac{|\tr(\gamma)|}{2}\right),\] 
    where $z \in \Hyp$ is a point on the axis of $\gamma$.

    Since \(\gamma\) is hyperbolic, \(\gamma^\sigma\) is either hyperbolic or elliptic. If \(\gamma^\sigma\) is hyperbolic with displacement \(\ell'\), we have
    \begin{align*}
    2\arcosh\left(\frac{|\sigma(\tr( \gamma))|}{2}\right) =\ell' \leq ~  & d(F(z),\gamma^\sigma  F(z)) \\
                                                   =  ~ & d(F(z),F(\gamma z)) \\
                                                  \leq ~ &  (1-\delta) \ell \\
                                                   =  ~  & (1-\delta)2\arcosh\left(\frac{|\tr(\gamma)|}{2}\right).
    \end{align*}
    Therefore, we obtain
    \begin{equation}\label{lemma 3.1 eq6}
        |\sigma(\tr(\gamma ))|  \leq 2\cosh\left((1-\delta)\arcosh\left(\frac{|\tr(\gamma)|}{2}\right)\right). 
    \end{equation}

Since \(\arcosh(x)=\log(x+\sqrt{x^2-1})\) for all real \(x \ge 1,\) we obtain that for all \(x \ge 2\) we have \(\arcosh(x/2) \leq \log(x).\) 

Hence, 
\begin{align*}
\cosh((1-\delta)\arcosh(\nicefrac{|\tr(\gamma)|}{2})) &\leq \cosh((1-\delta)\log(|\tr(\gamma)|))<|\tr(\gamma)|^{1-\delta};\\
 |\sigma(\tr(\gamma ))| &< 2|\tr(\gamma)|^{1-\delta}.
\end{align*}

If \(\gamma^\sigma\) is elliptic, we have a stronger inequality \[|\sigma(\tr(\gamma))|<2<2|\tr(\gamma)|^{1-\epsilon}\]
for all \(0<\epsilon<1\). 

This proves \eqref{strict bound} and implies \eqref{eq norm bound}:
\[
|\mathrm{N}_{K \mid \Q}(\tr(\gamma))| = |\tr(\gamma) \cdot \sigma(\tr(\gamma)) \cdot \prod_{j=3}^d \sigma_j(\tr(\gamma))| < 2^{d-1}|\tr(\gamma)|^{2-\delta}.
\]
\end{proof}

The following lemma can be seen as a generalization of the bounded clustering property of \cite{LuoSarnak94} to semi-arithmetic Fuchsian groups with arithmetic dimension \(2\) which admit a modular embedding. It gives a weaker form of bounded 
clustering, which is still sufficient for our main applications. 

\begin{lemma} \label{lem B-C}
There exist \(0<\delta<1\) and $C>0$ depending on \(\Gamma\) such that for any $N > 2$ we have
\begin{align*}
    \#(\cL(\Gamma)\cap[N-1, N]) \leq CN^{1-\delta}.
\end{align*}
Moreover, if $\Gamma$ is derived from a quaternion algebra, then any two different traces in $\cL(\Gamma)\cap[2, T]$ are $cT^{\delta-1}$ separated, with a constant $c>0$ depending only on $d$.
\end{lemma}

\begin{proof} First assume that $\Gamma$ is derived from a quaternion algebra.
Consider the set \(\{\mathrm{id},\sigma,\sigma_3,\ldots,\sigma_d\}\) of the  embeddings of \(K\). We recall that \(|\sigma_j(\tr(\gamma))|<2\) for any non-identity element \(\gamma \in \Gamma^{(2)}\) and for all \(j=3,\ldots,d\). Let \(s\neq t \in (2,T]\) be two traces of elements of $\Gamma$. By inequality \eqref{strict bound} from Lemma~\ref{norm bound} applied to $|\sigma(s)|$ and $|\sigma(t)|$, we have 
\[1 \leq |\mathrm{N}_{K \mid \Q}(t-s)| \leq |t-s|(|\sigma(t)|+|\sigma(s)|)4^{d-2} <C_1|t-s|T^{1-\delta}\]
with \(C_1 = 4^{d-1} > 0\). These inequalities show that \(t\) and \(s\) are \(cT^{\delta-1}\) separated with \(c=\frac{1}{C_1}>0.\)

In general, similar to \cite[Lemma~2.1]{LuoSarnak94}, we consider all the embeddings of the field $L = \Q(\cL(\Gamma))$ into $\mathbb{C}$ which we denote by $\tau_1 = \id, \tau_2, \ldots, \tau_n$, $n = [L:\Q]$. We order them so that the first $k$ embeddings extend the embedding $\sigma_1 = \id$ of the subfield $K$ of $L$, and the next $k$ embeddings extend $\sigma_2 = \sigma$. Since $\tr(\gamma)^2 = \tr(\gamma^2) + 2$ we have \(x^2-2 \in \cL(\Gamma^{(2)})\) for any \(x \in \cL(\Gamma)\), in particular, \(x^2 \in K\). 

Suppose there are more than $2^{2k-1}(M+1)$ points of $\cL(\Gamma)$ in the interval $[N-1, N]$, with $M$ a positive integer. Then by the pigeonhole principle there are at least two different elements $t,s\in \cL(\Gamma)$ and some $\epsilon_i \in \{-1,1\}$,  $i = 2,\ldots,2k$, such that $0 < |t - s| \leq \frac1M$, $\tau_i(t) = \epsilon_i t$, $\tau_i(s) = \epsilon_i s$ for $2\le i \le k,$ and $\tau_i(t) = \epsilon_i \sqrt{\sigma_2(t^2)}$, $\tau_i(s) = \epsilon_i \sqrt{\sigma_2(s^2)}$ for \(k<i\le 2k\). 

Let $0\neq t-s \in \cO_L$ be two such elements. For $1\leq i \leq k$, we have $|\tau_i(t-s)| \leq \frac1M$ by the assumption. For the next $k$ embeddings $k< i \leq 2k$,  $|\tau_i(t-s)| \leq 4N^{1-\delta}$ by Lemma~\ref{norm bound} applied to $\gamma^2$ and our assumption on $\tau_i$. And for the remaining for $i > 2k$, we have $|\tau_i(t-s)| \leq 4$  (we assume that $\tau_i(\cL(\Gamma)) \subset [-2,2]$ for $i > 2k$, which follows from $\tr(\gamma)^2 = \tr(\gamma^2) + 2$ and $\sigma_j(\cL(\Gamma^{(2)})) \subset [-2,2]$, for $j >2$). Hence we show that  
$$1 \leq |\mathrm{N}_{L \mid \Q}(t-s)| \leq \frac1{M^k} (4N)^{(1-\delta)k}4^{n-2k},$$
which implies that $M \leq C N^{1-\delta}$ with $C = C(\Gamma) >0$.
\end{proof}

\begin{corollary} \label{cor lem B-C}
There exist \(0<\delta<1\) and \(C>0\) depending on \(\Gamma\) such that for $T>2$,     
     \begin{equation*}
        \#(\cL(\Gamma)\cap[2, T]) < C T^{2-\delta}.
    \end{equation*}
\end{corollary}
\begin{proof}
    We have
    $$ \#(\cL(\Gamma)\cap[2, T]) \leq \sum_{N = 3}^{T}\#(\cL(\Gamma)\cap[N-1, N]) < C T^{1-\delta} T.$$
\end{proof}

We now prove the first main result:

\begin{theorem}\label{thm1}
 All cocompact semi-arithmetic Fuchsian groups which admit a modular embedding and have arithmetic dimension at most \(2\) have exponential growth of the mean multiplicity in the length spectrum.   
\end{theorem}
\begin{proof}
First assume that  \(\Gamma\) is a non-arithmetic group as above. 

Let \(\mathcal{N}'(\ell)\) be the number of distinct lengths bounded by \(\ell\) of closed geodesics in the orbifold \(\Gamma \backslash \Hyp.\)  Each length produces a unique positive trace of a hyperbolic element of \(\Gamma\) bounded by \(2\cosh({\ell}/{2})\), hence by Corollary~\ref{cor lem B-C} we have 
\begin{equation}
    \mathcal{N}'(\ell) \leq C\left(2\cosh\left(\frac{\ell}{2}\right)\right)^{2-\delta}.
\end{equation}

Since \(\cosh(\frac{\ell}{2})\leq e^{\frac{\ell}{2}}\), we obtain
\begin{equation}
     \mathcal{N}'(\ell) \leq C'e^{(1-\delta')\ell}
\end{equation}
for some constants \(C'>0\) and \(0<\delta'<\frac{1}{2}.\)
By the prime geodesic theorem (cf.~\cite[Theorem~3.4]{Sarnak80}), the number \(\mathcal{N}(\ell)\) of closed geodesics of \(\Gamma \backslash \Hyp\) counted with multiplicity is given by
\begin{align*} 
\mathcal{N}(\ell) \sim \dfrac{e^\ell}{\ell} & \text{ when } \ell \to \infty. \end{align*}

By differentiating the integral formula for $\langle g(\ell) \rangle$ we see that the asymptotic behavior of the mean multiplicity is determined by the average derivatives of the counting functions (cf. \cite[Section~3.2]{Bol93_long}):
$$
\dfrac{\mathcal{N}(\ell)}{\ell} \sim \langle g(\ell) \rangle \dfrac{\mathcal{N'}(\ell)}{\ell}\text{, when }\ell\to\infty.
$$

Thus the mean multiplicity of closed geodesics of length at most \(\ell\) grows as
\begin{align*}
    \langle g(\ell) \rangle \geq c \dfrac{e^{\delta' \ell}}{\ell} 
\end{align*}
for sufficiently large $\ell$, where the constants \(c>0\) and \(0<\delta'<\frac{1}{2}\) depend on  \(\Gamma\).

If $\Gamma$ is arithmetic the result is well known. It follows from the bounded clustering property of Luo and Sarnak \cite[Lemma~2.1]{LuoSarnak94} applied the same way that we use our Lemma~\ref{lem B-C} (see also \cite{Bol93} for more details).
\end{proof}

Let now \(S\) be a complete non-compact hyperbolic surface with \(n\) cusps. We recall that \(S\) has a \emph{compact core}, i.e. there exists a compact subsurface \(\Omega \subset S\) with \(n\) boundary components, formed by horocycles, containing all the topology of \(S\). More precisely, there exists a retract \(r: S \to \Omega\) whose induced homomorphism of the fundamental groups is an isomorphism.
 
\begin{lemma}\label{lemma: extension map}
    Let \(S\) be a complete non-compact hyperbolic surface with finite area and compact core \(\Omega\), and let \(X\) be a metric space of curvature \(\leq \kappa\) (in the sense of Alexandrov) for some \(\kappa\leq 0\). For any Lipschitz map \(f:S \to X\), there exists a Lipschitz map \(g: S \to X\) such that \(g|_\Omega  = f |_\Omega\) and the Lipschitz constant of \(g\) is equal to the Lipschitz constant of \(f |_\Omega\). In particular, \(g_*=f_*:\pi_1(S) \to \pi_1(X).\) 
\end{lemma}
\begin{proof}
Since \(X\) has curvature \(\leq \kappa\) for some \(\kappa \leq 0\) and \(S\) has curvature \(\leq -1\) in the sense of Alexandrov, we can apply the generalized Kirszbraun theorem \cite[Theorem~A]{LangSchr97} in order to obtain an extension \(g\) of \(f \mid \Omega\) with the same Lipschitz constant.
\end{proof}

\begin{lemma}\label{lem norm bound 2}
Let \(\Gamma\) be a non-cocompact semi-arithmetic Fuchsian group of arithmetic dimension \(2\) which admits a modular embedding. Then there exists a  constant \(0<\delta<1\) depending on $\Gamma$ such that 
\[|\sigma(\tr(\gamma))| \leq 2|\tr(\gamma)|^{1-\delta}\]
for any hyperbolic element \(\gamma \in \Gamma^{(2)}.\)    
\end{lemma}
\begin{proof} \textbf{1.} Let us first assume that the group $\Gamma$ is torsion-free. By restricting, if necessary, to the subgroup $\Gamma^{(2)}$ we can assume that $\Gamma$ is derived from a quaternions algebra.
Consider the hyperbolic surface \(S=\Gamma \backslash \Hyp\) and a lattice \(\Delta < \SL(2,\R)\times\SL(2,\R)\) with \(\rho(\gamma):=(\gamma,\gamma^{\sigma}) \in \Delta\) for all \(\gamma \in \Gamma.\) 

Let \(\Omega \subset S\) be a compact core of \(S\) and let $F:\Hyp \to \pm \Hyp$ be a holomorphic or anti-holomorphic map such that \(F(\gamma \cdot z)=\gamma^\sigma \cdot F(z)\), for any \(z \in \Hyp\) and \(\gamma \in \Gamma\), provided by the modular embedding of $\Gamma$. Consider the metric space \(X=\Delta \backslash(\pm\Hyp\times\pm\Hyp)\), which is locally isometric to \(\Hyp\times\Hyp\) equipped with the product metric. We have an embedding \(G:S \to X\)  given by 
$$G(\Gamma z)=\Delta(z,F(z)).$$  
When  restricted to \(\Omega\), the map \(G\) is \(\sqrt{1+(1-\delta)^2}\)-Lipschitz for some \(0<\delta<1\) according to the Schwarz--Pick lemma (cf. the proof of Lemma~\ref{norm bound}). 

By applying Lemma~\ref{lemma: extension map}, we can find a \(\sqrt{1+(1-\delta)^2}\)-Lipschitz map \(G_1:S \to X\) which is homotopic to \(G\). 

Given any hyperbolic element \(\gamma \in \Gamma,\) we consider the closed geodesic \(\overline{\gamma} \subset S\) of length \(\ell\) given by the projection of the axis of \(\gamma\). By the definition of \(G\), the image \(G(\gamma)\) is freely homotopic to the projection on \(X\) of the axis of the hyperbolic element \(\rho(\gamma) \in \Delta\), whose length is given by \(\sqrt{\ell^2+\ell_{\sigma}^2}\),  where \(\ell_\sigma\) is the displacement of \(\gamma^{\sigma}.\) Since \(G\) is homotopic to \(G_1\), we conclude that \(G(\gamma)\) is homotopic to \(G_1(\gamma)\).  Hence, the displacement of \(\rho(\gamma)\) is smaller than the length of the closed curve \(G_1(\gamma)\). Now we use the fact that \(G_1\) is \(\sqrt{1+(1-\delta)^2}\)-Lipschitz in order to get the estimate

\begin{equation*}
    \ell^2+\ell_{\sigma}^2  \leq \big(1+(1-\delta)^2\big) \ell^2.
\end{equation*}

Hence, 
\begin{equation}\label{eq: comparison between discplacements}
 \ell_{\sigma}   \leq (1-\delta)\ell.
\end{equation}
 
There are two possibilities, the first one occurs when  \(\ell_\sigma =0\), i.e. \(\gamma^\sigma\) is elliptic, in which case the lemma follows immediately.  
If not, then \(\gamma^\sigma\) is hyperbolic and we have \(\ell_\sigma=2\arcosh\left(\frac{|\sigma(\tr(\gamma))|}{2}\right)\).  By \eqref{eq: comparison between discplacements}, we obtain
  \begin{align*}
 2\arcosh\left(\frac{|\sigma(\tr(\gamma))|}{2}\right)   \leq  2(1-\delta)\arcosh\left(\frac{|\tr(\gamma)|}{2}\right); \\
        |\sigma(\tr(\gamma ))|  \leq 2\cosh\left((1-\delta)\arcosh\left(\frac{|\tr(\gamma)|}{2}\right)\right). 
  \end{align*}
This is estimate \eqref{lemma 3.1 eq6} from  Lemma~\ref{norm bound} and the proof concludes in the same way.

\textbf{2.} It remains to consider the case when the group $\Gamma$ has torsion. Let us fix  a finite index torsion-free subgroup of the group  $\Gamma^{(2)}$ and denote it $\Gamma_1$. 

With the group \(\Gamma_1\) chosen, by \eqref{eq: comparison between discplacements} we have that \(\ell(\beta^\sigma) \leq (1-\delta)\ell(\beta)\) for some constant \(\delta>0\) and all hyperbolic elements \(\beta \in \Gamma_1\). Now we notice that this inequality can be extended to all hyperbolic \(\gamma \in \Gamma^{(2)}\). Indeed, given any such \(\gamma\), there exists a positive integer \(j\) such that \(\gamma^j \in \Gamma_1\). Since the action of \(\sigma\) preserves the products, we have \((\gamma^\sigma)^j = (\gamma^j)^\sigma\). By the previous argument, the length of the geodesic associated to $(\gamma^j)^\sigma$ satisfies inequality \eqref{eq: comparison between discplacements}. Hence for any hyperbolic $\gamma\in\Gamma^{(2)}$, we have 

\[\ell(\gamma^\sigma)=\frac {\ell((\gamma^j)^\sigma)}j \leq \frac{(1-\delta)\ell(\gamma^j)}j = (1-\delta)\ell(\gamma).\]
This is an extension of formula \eqref{eq: comparison between discplacements} to the group $\Gamma^{(2)}$ and the proof concludes as in part~1.
\end{proof}

We now deduce the second main result of the paper.

\begin{theorem}\label{thm2}
 All non-cocompact semi-arithmetic Fuchsian groups which admit a modular embedding and have arithmetic dimension at most \(2\) have exponential growth of the mean multiplicity in the length spectrum.   
\end{theorem}

\begin{proof}
Let $\Gamma$ be a non-cocompact Fuchsian group satisfying the assumptions. By Lemma~\ref{lem norm bound 2}, the group $\Gamma$ satisfies the weak bounded clustering property of Lemma~\ref{lem B-C}. The rest of the argument follows the line of the proof of Theorem~\ref{thm1}.
\end{proof}

\section{Examples}\label{290625.1}

In this section, we give examples of groups to which Theorems \ref{thm1} and \ref{thm2} apply, i.e., semi-arithmetic groups of arithmetic dimension $2$ admitting modular embedding. The first family of examples, coming from triangle groups, is finite and includes both cocompact and non-cocompact groups. The second family contains infinitely many commensurability classes. This family is comprised of Veech groups, hence non-cocompact.

\subsection{Triangle group examples}

Consider a geodesic triangle in $\bH$ with angles %submultiples of $\pi$.
$\pi/a, \pi/b, \pi/c$, where $a,b,c$ are positive integers (possibly infinite). Such a triangle always exists and is unique up to isometry as long as $1/a + 1/b + 1/c <1$. Reflections across the geodesic lines supporting the sides of the triangle generate a discrete subgroup of the full group of isometries of $\bH$. The index $2$ subgroup of orientation-preserving isometries is a Fuchsian group of finite coarea and presentation $\langle x,y,z \mid x^a = y^b = z^c = xyz\rangle$, called a \emph{triangle group}.%, where $\pi/a, \pi/b, \pi/c$ are the angles of the geodesic triangle.

Arithmetic properties of triangle groups were first studied in \cite{Takeuchi77}, in particular showing that these groups are semi-arithmetic. Furthermore, in \cite{Cohen90} and \cite{Ricker02} all triangle groups are proved to admit modular embedding. In \cite{Nugent17} Nugent and Voight generalise the methods in \cite{Takeuchi77}, proving there exist at most finitely many triangle groups of a given arithmetic dimension (see also \cite{BCDT25}). They also describe an algorithm to list all such groups of a given arithmetic dimension. Including both cocompact and non-cocompact groups, precisely $164$ triangle groups have arithmetic dimension $2$, and a complete list can be found in their paper.

\subsection{Veech group examples}

First let us briefly recall the definition of a Veech group. An introduction to the subject can be found in several sources, for instance \cite{Earle97}, \cite{McMullen03} and \cite{Veech89}.

Let $X\in \cT_g$ be a marked Riemann surface and let $q$ be a holomorphic quadratic differential on $X$ of norm $\norm{q} = 1$, where $\norm{q} = \int_X |q|$. Away from the zeroes of $q$, one can find coordinate charts $\zeta: U \to V$ with respect to which $q$ has the form $\rd \zeta^2$. The transition functions of such charts are half-translations: $z \mapsto \pm z + c$, for some $c \in \C$. Together, these \emph{preferred coordinate charts} define a \emph{flat structure} on $X$.

An element $A \in \SL(2,\R)$ induces a new set of coordinate charts $\{A\zeta: U \to A(V)\}$ whose transition functions are half-translations. In particular, $A$ defines a new point $X_A \in \cT_g$ (along with a new holomorphic quadratic differential). Note that if $A \in\SO(2)$, then $\id: X \to X_A$ is holomorphic, which means they represent the same point in $\cT_g$. The map $A \mapsto X_A$ descends to an embedding
\begin{align*}
  F: \bH \cong \SO(2)\backslash \SL(2,\R) \hookrightarrow \cT_g,
\end{align*}
which is isometric with respect to the Poincaré metric on $\bH$ and the Teichmüller metric on $\cT_g$. We shall denote the image of $F$ by $\Delta_q$ and call it the \emph{Teichmüller disc} of $q$.

The isometric action of the Teichmüller modular group $\Mod(X)$ on $\cT_g$ defines the stabiliser group of $\Delta_q$:
\begin{align*}
    \Stab(\Delta_q) = \{h\in\Mod(X) \mid h(\Delta_q) = \Delta_q\}.
\end{align*}
Let $N$ be the subgroup of $\Stab(\Delta_q)$ consisting of those elements that fix $\Delta_q$ pointwise. In particular, elements of $N$ are (conformal) automorphisms of $X$, and so $N$ is finite. Note that each $h \in \Stab(\Delta_q)$ corresponds to an isometric automorphism of $\bH$, and thus to an element of $\PSL(2,\R)$. We may then realise the quotient $N\backslash \Stab(\Delta_q)$ as a subgroup of  $\PSL(2,\R)$, denoted $\PSL(X,q)$.

Alternatively, define $\Aff^+(q)$ to be the group of quasi-conformal self\-/homeomorphisms $g$ of $X$ that are affine with respect to the flat structure induced by $q$. This means that, for each $p \in X\setminus \{\text{zeros of }q\}$ there are preferred coordinate charts $\zeta_1$ around $p$ and $\zeta_2$ around $g(p)$ such that $\zeta_2 \circ g \circ \zeta_1^{-1}$ is an affine homeomorphism between domains of $\R^2$. The derivative of this affine homeomorphism is independent of $p$ and of the coordinate charts up to sign. Its determinant has norm 1, since $\norm{q}=1$. This induces a homomorphism $a : \Aff^{+}(q) \to \PSL(2,\R)$, taking an affine quasi-conformal homeomorphism $g$ to its derivative $a(g)$ with respect to the flat structure of $q$.

By \cite[Lemma 5.2]{Earle97}, $\Aff^{+}(q)$ projects isomorphically onto its image in $\Mod(X)$. Moreover, by \cite[Theorem 1]{Earle97} $\Aff^{+}(q)$ coincides with $\Stab(\Delta_q)$ and the kernel of $g \mapsto a(g)$, with $N$. So $\PSL(X,q)$ is the image of $\Aff^{+}(q)$ under $a$.

The group $\PSL(X,q)$ is a Fuchsian group, i.e., it is discrete in $\PSL(2,\R)$ \cite{Veech89}.

\begin{definition}
  When $\PSL(X,q)$ is a lattice in $\PSL(2,\R)$, we call it the \emph{Veech group} of $(X,q)$.
\end{definition}

In this case, the map $\bH \xlongrightarrow{F} \cT_g \longrightarrow \cM_g$ factors through the quotient $V = \PSL(X,q) \backslash \bH$, giving the \emph{Teichmüller curve}:
\begin{align*}
  f : V \to \cM_g.
\end{align*}

Let us assume henceforth that $\Gamma$ is a Veech group of some $(X,q)$. In \cite{Veech89}, Veech proved that $\Gamma$ is always non-cocompact. 

Some arithmetic properties of $\Gamma$ are investigated in \cite{McMullen03}. The ones of interest to us are the fact that the trace field $\Q(\tr(\Gamma)) \subset \R$ is a number field and that the traces of $\Gamma$ are algebraic integers \cite[Theorem 5.2]{McMullen03}.

In \cite[Section 9]{McMullen03}, McMullen constructs a family of examples of Teichmüller curves in $\cM_2$. Briefly, given a rational Euclidean polygon $P$, one assembles (finitely many) rotated copies of $P$ into a polygon whose sides fall into pairs of parallel segments. Identifying each pair of parallel sides via a Euclidean translation yields a Riemann surface $X$ (see, for example, \cite[Section 1.3]{Masur02}). The holomorphic 1-form $\rd z$ of $\C$ descends to a holomorphic 1-form $\omega$ on $X$, giving us a pair $(X,\omega^2)$ as before and, consequently, a Fuchsian group $\PSL(X,\omega)$ (which may or may not be a lattice, i.e., a Veech group). Using certain L-shaped polygons, McMullen constructs a family of Veech groups $\Gamma_d$ %realising the real quadratic field $\Q(\sqrt{d})$ as its trace field.
with trace field $\Q(\sqrt{d})$, for all square-free positive integers $d$. As can be seen from the construction (see also \cite[Theorem 10.3]{McMullen03}), the invariant trace field of $\Gamma_d$ contains $\Q$ properly, so $k\Gamma_d = \Q(\sqrt{d})$. In particular, the groups $\Gamma_d$ are (non-cocompact) semi-arithmetic of arithmetic dimension $2$. In fact, they satisfy an even stronger property: as shown in \cite[Theorem 10.1]{McMullen03}, the groups $\Gamma_d$ admit modular embedding.

Notice from the properties listed above that all but finitely many of these groups are not commensurable to a triangle group. These were the first examples of groups admitting modular embedding that are neither arithmetic nor commensurable to a triangle group, answering a question posed in \cite{SW00}.

\section{Quantum chaos}

A big part of the interest in geodesic length spectra comes from their connection with quantum mechanics. In this section we briefly review the connection and make some comments related to our results. We refer to \cite{Bol93_long}, \cite{Bog97}, and \cite{Sarnak95} for extensive surveys of this topic.

Consider the motion of a free point particle on an oriented surface of constant negative curvature $S = \Gamma\backslash\Hyp$. This is a classical Hamiltonian system whose trajectories are defined by the geodesic flow on the cotangent bundle $T^*(S)$. It is well known that the geodesic flow is \emph{chaotic}, meaning that when restricted to a fixed energy level (a time-independent Hamiltonian) the flow is ergodic and has almost everywhere positive Lyapunov exponents \cite{AA68}. Its trajectories are very sensitive to the initial data and unpredictable in this sense. At the opposite extreme we have the \emph{integrable systems} whose time evolution in the phase space $T^*(S)$ takes place on invariant tori. These systems are the most regular and predictable. 

The basic problem of \emph{quantum chaos} is to characterize universal properties of quantum systems that reflect the regular or chaotic features of the underlying classical dynamics.

The quantum mechanics of the Hamiltonian $H$ on a surface $S$ is described by the wave functions of the particle $\psi_j$ and energy levels $\lambda_j$. 
These are given by the solutions of the eigenvalue problem for the Laplace--Beltrami operator $\Delta$ on $S$, as it is well known that $\hslash^2\Delta$ corresponds to a quantization of $H$:
$$-\hslash^2\Delta \psi_j = \lambda_j\psi_j.$$
The eigenvalues (resp. $L^2$-eigenvalues, if the surface $S$ is not compact) form a discrete set $0=\lambda_0 < \lambda_1 \leq \lambda_2\ldots$ and the eigenfunctions $\psi_j$ can be chosen to form an orthonormal basis (see \cite[Section~4]{Sarnak95}). It follows that for the geodesic flow the 
semi-classical limit $\hslash \to 0$ corresponds to the large eigenvalue limit $\lambda_j \to \infty$. Hence the basic problem here is to understand the behavior of $\lambda_j$ and $\psi_j$ as $j\to\infty$ and especially their relation to the orbits of the geodesic flow. In general, the relation between the classical length spectrum and quantum energy levels is described by the trace formulae. 

There are different statistics that may be used to measure the fine structure of the distribution of the energy levels $\lambda_j$. One of the first choices is the number variance $\Sigma^2(\lambda, L)$ which measures the average derivation from the expected number of the number of energy levels in the intervals of length $L$. %It can be defined by the formula $$\Sigma^2(\lambda, L) =  \frac{1}{\lambda}...$$
The basic models for the distribution of the $\lambda$'s relevant to the spectral problems are
\begin{itemize}
    \item[(i)] \emph{Poisson:} The $\lambda_j$ are random meaning that $\mu_j = \lambda_{j+1}-\lambda_j$ has a Poisson distribution and $\mu_j$'s are independent from each other. This distribution has number variance $\Sigma^2(L) \sim L$.
    \item[(ii)] \emph{Random Matrix Theory:} The $\lambda_j$ are the eigenvalues of a random (symmetric or hermitian) matrix, in which case $\Sigma^2(L) \sim \log(L)$ and the energy level spectrum is rigid.
\end{itemize}
It is generally expected that for integrable systems the eigenvalues follow the Poisson behavior while for chaotic systems they follow the distributions of random matrix ensembles.

In \cite{LuoSarnak94}, Luo and Sarnak obtained a lower bound on the average number variance for arithmetic surfaces. They proved that 
$$
\frac{1}{L}\int_0^L \Sigma^2(\lambda,\xi) d\xi \gg \frac{\sqrt{\lambda}}{\log^2(\lambda)},\text{ for } \frac{\sqrt{\lambda}}{\log(\lambda)} \ll L \ll \sqrt{\lambda}. 
$$
For $L \sim \sqrt{\lambda}/\log(\lambda)$ which is just inside the parameter range it gives the average number variance asymptotically bounded below by $L/\log(L)$. 
Surprisingly, this result shows that the quantum mechanics on arithmetic surfaces is consistent with the Poisson distribution and thus comes close to the integrable systems. 

The proof of Luo and Sarnak uses the bounded clustering property and the Selberg trace formula. On the qualitative level the non-rigidity of the spectrum of arithmetic quantum systems is attributed to the large multiplicities in the length spectrum (cf. \cite{Sarnak95}). Bolte \cite[Section~4]{Bol93_long} and Bogomolny et al \cite[Section~6]{Bog97} considered a simple physical model which allows to estimate the number variance $\Sigma^2(\lambda, L)$ directly from the growth of mean multiplicity. Their method applied to arithmetic surfaces shows agreement with the Poisson distribution but in a different range of parameters compared to the results in \cite{LuoSarnak94}. This is consistent with the fact that the B-C property is stronger than EGMM. On the other hand, the physical model allows to consider the effect of EGMM of semi-arithmetic groups. This was done by Bogomolny et al in their paper, as at that time the experimental data already provided a strong evidence that some non-arithmetic systems might have large multiplicities in their length spectra. However, it appears that the exponent in the exponential growth being smaller than $\frac{1}{2}$ does not allow to perform the required approximations in the physical model. The authors also point out that the energy levels of non-arithmetic groups computed numerically show strong corelation with the rigid spectrum of the random matrix model \cite[Section~10]{Bog97}. 

Our results indicate that it would be interesting to find some particular features in the spectrum of a quantized system that distinguish semi-arithmetic  surfaces with EGMM from the other non-arithmetic surfaces. It is also appealing to understand if, as suggested by our estimates, the arithmetic dimension $2$ plays a special role.

Multiplicities in length spectra of arithmetic hyperbolic
$3$-orbi\-folds with the connection to $3$-dimensional chaotic quantum systems were studied by Marklof \cite{Markl96}. He established strong exponential growth of mean multiplicities in the arithmetic case, although for cocompact groups the proof is conditional on a conjecture on the number of gaps in the length spectra. In a subsequent paper with Aurich \cite{AuMarkl96}, the authors discovered experimentally that the tetrahedron $T_8$ group $\Gamma$, which is known to be non-arithmetic, also shows exponential growth of mean multiplicities. It would be interesting to have a rigorous proof of this phenomenon. 
The group of the tetrahedron  $T_8$ is defined over the field $\Q\left(\sqrt{-(1+\sqrt{2})(\sqrt{2}+\sqrt{5})}\right)$ which has $2$ complex and $4$ real places, and the associated quaternion algebra is ramified at the real places (see \cite[Section~13.1]{Maclachlan03}). It can also be checked that the elements of $\Gamma$ have integral traces. Although the notion of semi-arithmeticity is not well developed in dimensions higher than two, these properties bring the Aurich--Marklof example somewhat close to the groups considered in our paper. 

\medskip

\noindent
\textbf{Data Availability.} Data sharing is not applicable to this article as no datasets were generated or analysed during the current study.

\medskip
\noindent
\textbf{Declarations}

\medskip

\noindent
\textbf{Conflict of interest.} The authors have no Conflict of interest to declare that are relevant to the content of this article.

\bibliographystyle{alpha}
\bibliography{Bibliography}

\begin{thebibliography}{BCDTP25}

\bibitem[AA68]{AA68}
V.~I. Arnold and A.~Avez.
\newblock {\em Ergodic problems of classical mechanics}.
\newblock W. A. Benjamin, Inc., New York-Amsterdam, 1968.
\newblock Translated from the French by A. Avez.

\bibitem[AM96]{AuMarkl96}
R.~Aurich and J.~Marklof.
\newblock Trace formulae for three-dimensional hyperbolic lattices and application to a strongly chaotic tetrahedral billiard.
\newblock {\em Phys. D}, 92(1-2):101--129, 1996.

\bibitem[AS88]{AuSt88}
R.~Aurich and F.~Steiner.
\newblock On the periodic orbits of a strongly chaotic system.
\newblock {\em Phys. D}, 32(3):451--460, 1988.

\bibitem[BC05]{BergeroClozel05}
Nicolas Bergeron and Laurent Clozel.
\newblock {\em Spectre automorphe des vari{\'e}t{\'e}s hyperboliques et applications topologiques}, volume 303 of {\em Ast{\'e}risque}.
\newblock Paris: Soci{\'e}t{\'e} Math{\'e}matique de France, 2005.

\bibitem[BCDTP25]{BCDT25}
Mikhail Belolipetsky, Gregory Cosac, Cayo Dória, and Gisele Teixeira~Paula.
\newblock Geometry and arithmetic of semi-arithmetic {F}uchsian groups.
\newblock {\em J. Lond. Math. Soc.}, 111(2):e70087, 2025.

\bibitem[BGGS92]{Bog92}
E.~B. Bogomolny, B.~Georgeot, M.-J. Giannoni, and C.~Schmit.
\newblock Chaotic billiards generated by arithmetic groups.
\newblock {\em Phys. Rev. Lett.}, 69(10):1477--1480, 1992.

\bibitem[BGGS97]{Bog97}
E.~B. Bogomolny, B.~Georgeot, M.-J. Giannoni, and C.~Schmit.
\newblock Arithmetical chaos.
\newblock {\em Phys. Rep.}, 291(5-6):219--324, 1997.

\bibitem[Bol93a]{Bol93}
Jens Bolte.
\newblock Periodic orbits in arithmetical chaos on hyperbolic surfaces.
\newblock {\em Nonlinearity}, 6(6):935--951, 1993.

\bibitem[Bol93b]{Bol93_long}
Jens Bolte.
\newblock Some studies on arithmetical chaos in classical and quantum mechanics.
\newblock {\em Internat. J. Modern Phys. B}, 7(27):4451--4553, 1993.

\bibitem[BS04]{Bog04}
E.~B. Bogomolny and C.~Schmit.
\newblock Multiplicities of periodic orbit lengths for non-arithmetic models.
\newblock {\em J. Phys. A}, 37(16):4501--4526, 2004.

\bibitem[CW90]{Cohen90}
Paula Cohen and J{\"u}rgen Wolfart.
\newblock Modular embeddings for some non-arithmetic {F}uchsian groups.
\newblock {\em Acta Arithmetica}, 56(2):93--110, 1990.

\bibitem[EG97]{Earle97}
Clifford~J. Earle and Frederick~P. Gardiner.
\newblock Teichm{\"u}ller disks and {V}eech’s {F}-structures.
\newblock {\em Extremal Riemann surfaces}, pages 165--189, 1997.

\bibitem[GL08]{GenLeu08}
Slavyana Geninska and Enrico Leuzinger.
\newblock A geometric characterization of arithmetic {F}uchsian groups.
\newblock {\em Duke Math. J.}, 142(1):111--125, 2008.

\bibitem[LS94]{LuoSarnak94}
Wenzhi Luo and Peter Sarnak.
\newblock Number variance for arithmetic hyperbolic surfaces.
\newblock {\em Comm. Math. Phys.}, 161(2):419--432, 1994.

\bibitem[LS97]{LangSchr97}
Urs Lang and Viktor Schroeder.
\newblock Kirszbraun's theorem and metric spaces of bounded curvature.
\newblock {\em Geom. Funct. Anal.}, 7(3):535--560, 1997.

\bibitem[Mar96]{Markl96}
Jens Marklof.
\newblock On multiplicities in length spectra of arithmetic hyperbolic three-orbifolds.
\newblock {\em Nonlinearity}, 9(2):517--536, 1996.

\bibitem[McM03]{McMullen03}
Curtis~T. McMullen.
\newblock Billiards and {T}eichm\"{u}ller curves on {H}ilbert modular surfaces.
\newblock {\em J. Amer. Math. Soc.}, 16(4):857--885, 2003.

\bibitem[McM22]{McMullen22}
Curtis~T. McMullen.
\newblock Billiards, heights, and the arithmetic of non-arithmetic groups.
\newblock {\em Invent. Math.}, 228(3):1309--1351, 2022.

\bibitem[McM23]{McMullen23}
Curtis McMullen.
\newblock Triangle groups and {H}ilbert modular varieties.
\newblock {\em Preprint}, 2023.

\bibitem[MR03a]{maclachlan2003arithmetic}
Colin Maclachlan and Alan~W. Reid.
\newblock {\em The arithmetic of hyperbolic 3-manifolds}, volume 219.
\newblock Springer, 2003.

\bibitem[MR03b]{Maclachlan03}
Colin Maclachlan and Alan~W. Reid.
\newblock {\em The arithmetic of hyperbolic 3-manifolds}, volume 219 of {\em Graduate Texts in Mathematics}.
\newblock Springer-Verlag, New York, 2003.

\bibitem[MT02]{Masur02}
Howard Masur and Serge Tabachnikov.
\newblock Rational billiards and flat structures.
\newblock In {\em Handbook of dynamical systems}, volume~1, pages 1015--1089. Elsevier, 2002.

\bibitem[NV17]{Nugent17}
Steve Nugent and John Voight.
\newblock On the arithmetic dimension of triangle groups.
\newblock {\em Math. Comp.}, 86(306):1979--2004, 2017.

\bibitem[Ric02]{Ricker02}
Sabine Ricker.
\newblock Symmetric {F}uchsian quadrilateral groups and modular embeddings.
\newblock {\em Q. J. Math.}, 53(1):75--86, 2002.

\bibitem[Sar80]{Sarnak80}
Peter Sarnak.
\newblock {\em Prime geodesic theorems}.
\newblock ProQuest LLC, Ann Arbor, MI, 1980.
\newblock Thesis (Ph.D.)--Stanford University.

\bibitem[Sar95]{Sarnak95}
Peter Sarnak.
\newblock Arithmetic quantum chaos.
\newblock In {\em The {S}chur lectures (1992) ({T}el {A}viv)}, volume~8 of {\em Israel Math. Conf. Proc.}, pages 183--236. Bar-Ilan Univ., Ramat Gan, 1995.

\bibitem[SSW00]{SW00}
Paul Schmutz~Schaller and J{\"u}rgen Wolfart.
\newblock Semi-arithmetic {Fuchsian} groups and modular embeddings.
\newblock {\em J. Lond. Math. Soc., II. Ser.}, 61(1):13--24, 2000.

\bibitem[Tak75]{Takeuchi75}
Kisao Takeuchi.
\newblock A characterization of arithmetic {F}uchsian groups.
\newblock {\em Journal of the Mathematical Society of Japan}, 27(4):600--612, 1975.

\bibitem[Tak77]{Takeuchi77}
Kisao Takeuchi.
\newblock Arithmetic triangle groups.
\newblock {\em Journal of the Mathematical Society of Japan}, 29(1):91--106, 1977.

\bibitem[Vee89]{Veech89}
William~A. Veech.
\newblock Teichm{\"u}ller curves in moduli space, {E}isenstein series and an application to triangular billiards.
\newblock {\em Inventiones mathematicae}, 97:553--583, 1989.

\end{thebibliography}
\end{document}